\documentclass[11pt]{article}
\usepackage{amsmath}
\usepackage{amsfonts}
\usepackage{amsthm}
\usepackage{amssymb, setspace}
\usepackage{color,graphicx,epsfig,geometry,hyperref,fancyhdr}
\usepackage[T1]{fontenc}
\usepackage{comment}
\usepackage{float}
\usepackage{subfig}
\usepackage{commath}
\usepackage{bbm}
\usepackage{mathrsfs,fleqn}
\usepackage{pgfplots}
\pgfplotsset{compat=newest}
\newtheorem{theorem}{Theorem}[section]

\newtheorem{corollary}{Corollary}[section]
\newtheorem{lemma}{Lemma}[section]

\newcommand{\N}{\mathbb{N}}

\newcommand{\weakc}{\rightharpoonup}
\newcommand{\R}{\mathbb{R}}
\newcommand{\C}{\mathbb{C}}
\newcommand{\pdr}{\partial_r}

\newcommand{\e}{\text{ev}}
\newcommand{\p}{\text{pr}}

\usepackage{ulem}

\graphicspath{{paper-pics/}}

\begin{document}

\begin{flushleft}
\Large 
\noindent{\bf \Large Analysis of the transmission eigenvalue problem for biharmonic scattering considering penetrable scatterers}
\end{flushleft}

{\bf  \large Rafael Ceja Ayala}\\
\indent {\small School of Mathematical and Statistical Sciences, Arizona State University, Tempe, AZ 85287 }\\
\indent {\small Email:  \texttt{rcejaaya@asu.edu} }\\

{\bf  \large Isaac Harris}\\
\indent {\small Department of Mathematics, Purdue University, West Lafayette, IN 47907 }\\
\indent {\small Email: \texttt{harri814@purdue.edu} }\\

{\bf  \large Andreas Kleefeld}\\
\indent {\small Forschungszentrum J\"{u}lich GmbH, J\"{u}lich Supercomputing Centre, } \\
\indent {\small Wilhelm-Johnen-Stra{\ss}e, 52425 J\"{u}lich, Germany}\\
\indent {\small University of Applied Sciences Aachen, Faculty of Medical Engineering and } \\
\indent {\small Technomathematics, Heinrich-Mu\ss{}mann-Str. 1, 52428 J\"{u}lich, Germany}\\
\indent {\small Email: \texttt{a.kleefeld@fz-juelich.de}}\\

\begin{abstract}
\noindent In this paper, we provide an analytical study of the transmission eigenvalue problem in the context of biharmonic scattering with a penetrable obstacle. We will assume that the underlying physical model is given by an infinite elastic two--dimensional Kirchhoff--Love plate in $\mathbb{R}^2$, where the plate's thickness is small relative to the wavelength of the incident wave. In previous studies, transmission eigenvalues have been studied for acoustic scattering, whereas in this case, we consider biharmonic scattering. We prove the existence and discreteness of the transmission eigenvalues as well as study the dependence on the refractive index. We are able to prove the monotonicity of the first transmission eigenvalue with respect to the refractive index. Lastly, we provide numerical experiments to validate the theoretical work.
\end{abstract}

\section{Introduction}
In {this} paper, we study a transmission eigenvalue problem for biharmonic scattering with a penetrable medium. Wave scattering by penetrable obstacles plays an important role in many imaging and sensing applications, ranging from medical diagnostics and non-destructive testing in engineering to geophysical exploration. In these settings, the obstacle does not reflect all the incoming energy; instead, part of the wave is transmitted into the medium, leading to more complex interactions and interesting theory. A central object in the analysis of such problems are transmission eigenvalues, spectral values for non-trivial incident waves that produce no scattering. These values depend on both the material properties and the geometry of the obstacle, and carry useful information for inverse problems. Transmission eigenvalues have been a very active field of investigation in the area of inverse scattering. This is due to the fact that these eigenvalues can be recovered from the far-field data (see for example \cite{far field data,anisotropic}) and can be used to determine defects in a material \cite{TA,cavities,electro,mypaper1,te-cbc2,te-cbc3}. Although most of the classical theory on transmission eigenvalues focuses on second-order equations like the Helmholtz equation, many physical systems, especially those involving elasticity, thin plates, or surface effects, are better modeled using biharmonic operators. These arise naturally in the study of biharmonic wave scattering in elastic materials and plates, and have applications in modeling thin films, coated structures, or elastic shells. Yet, the corresponding transmission eigenvalue problem for biharmonic waves has received far less attention. In general, in acoustic scattering one can prove that the  transmission eigenvalues depend monotonically on the physical parameters, which implies that they can be used as a target signature for non-destructive testing. We show that the same happens in the problem, where one has biharmonic scattering. We will show that by having information or knowledge of the transmission eigenvalues, one can retrieve information about the material properties of the scattering object. Another reason one studies these eigenvalue problems, {is their non-linear and non-self-adjoint nature}. This makes them mathematically challenging to study. We refer to \cite{BS} for a survey on the study of transmission eigenvalue problems and to \cite{clamedTE} for a study of transmission eigenvalue problem for clamped boundary conditions.

Deriving accurate numerical algorithms to compute the transmission eigenvalues is an active field of study see for example \cite{spectraltev1,spectraltev2,fem-te,CMS,GP,numerical,mfs-te,eig-FEM-book} but most of the numerical algorithms are done with acoustic waves. In this paper, we consider transmission eigenvalues when we have biharmonic scattering and we will recover them numerically using separation of variables and boundary integral equations. 

The remainder of the paper is organized as follows. We will derive the transmission eigenvalue problem under consideration from the direct scattering problem in Section \ref{theproblem}. Next, in Section \ref{discrete} we prove that the transmission eigenvalues form a discrete set in the complex plane as well as provide an example via separation of variables to prove that this is a non-selfadjoint eigenvalue problem. Then in Section \ref{exist}, we prove the existence of infinitely many real transmission eigenvalues as well as study the dependence on the refractive index. Numerical examples, using separation of variables are given in Section \ref{numerics} to validate the analysis presented in the earlier sections. Further, numerical results are given using boundary integral equations. A short summary and an outlook is given in Section \ref{summary}.

\section{Formulation of the problem}\label{theproblem}

This work is concerned with the spectral analysis of a transmission eigenvalue problem arising from the scattering of biharmonic waves by penetrable obstacles in two dimensions. Specifically, we aim to connect the framework in transmission eigenvalues to the direct scattering problem governed by the biharmonic operator. Let $D \subset \mathbb{R}^2$ be a bounded domain with analytic boundary $\partial D$, representing a penetrable inclusion within an otherwise homogeneous medium. The scatterer is given by the support of the contrast, where we assume that the `refractive index' $n \in L^{\infty}(\R^2)$ such that $\overline{D}=$sup$(n-1)$. Now, we assume that the scattering medium is illuminated by a incident plane wave of the form $u^i(x, \hat{y}) = \text{e}^{\text{i}k x \cdot \hat{y}}$, where $k > 0$ denotes the wavenumber, $x \in \mathbb{R}^2$ is the spatial variable, and $\hat{y} \in \mathbb{S}^1$ is the direction of propagation.

The incident field $u^i$ satisfies the free-space biharmonic `Helmholtz' equation
\begin{align}
\Delta^2 u^i-k^4u^i=0 \quad \text{ in } \quad \mathbb{R}^{2}. \label{direct1}
\end{align}
The presence of the scatterer alters the propagation of the wave, giving rise to a scattered field $u^s(x,\hat{y})\in H^2_{\text{loc}}(\mathbb{R}^{2})$ that satisfies 
\begin{align}
\Delta^2 u^s-k^4 n(x)u^s=-k^4 \big(1-n(x)\big) u^i \quad & \text{in} \hspace{.2cm}  \mathbb{R}^2 . \label{direct2}
\end{align}
The transmission conditions across the boundary $\partial D$ enforce continuity conditions on the total field denoted $u=u^i+u^s$ which are given by 
\begin{align}
[\![ u ]\!] \big|_{\partial D} = [\![ \partial_\nu u ]\!] \big|_{\partial D} = [\![ \Delta u ]\!] \big|_{\partial D} = [\![ \partial_\nu \Delta u ]\!] \big|_{\partial D} = 0 \quad \text{on} \quad \partial D.
\label{eq:trans_conditions}
\end{align}
Here, $\nu$ denotes the outward unit normal vector on the boundary $\partial D$, and the notation 
$$ [\![\phi ]\!] \big|_{\partial D} := (\phi_{+} - \phi_{-}) \big|_{\partial D} $$ 
where `+' represents the trace taken from $\mathbb{R}^2 \setminus \overline{D}$ and the `$-$' represents the trace taken from $D$ for a given function $\phi$. Now, to ensure physical admissibility, the scattered field and its Laplacian are required to satisfy the Sommerfeld radiation condition at infinity
\begin{equation} \label{direct4}
\lim_{r \to \infty} \sqrt{r}(\pdr{u^s} - \text{i} k u^s) = 0,  \quad \text{ and } \quad \lim_{r \to \infty} \sqrt{r}(\pdr{\Delta u^s} - \text{i} k \Delta u^s) = 0
\end{equation}
for $r :=|x|$, where we assume that the limits are satisfied uniformly for all directions on the unit sphere $\hat{x}:=x/|x|\in \mathbb{S}^1$ (see \cite{novelbi,Inverserandom} for details).

A useful decomposition of the scattered field outside of $D$ is obtained by exploiting the operator identity $(\Delta^2 - k^4) = (\Delta - k^2)(\Delta + k^2)$ and introducing two auxiliary fields defined by the scattered field. To this end, just as in \cite{biwellposed,optiforinversebiha,novelbi} we define the auxiliary fields
$$ u_\p := -\frac{1}{2k^2} (\Delta - k^2) u^s \quad \text{and} \quad u_\e := \frac{1}{2k^2} (\Delta + k^2) u^s,$$
where the scattered field $u^s = u_\p +u_\e$. We have that $u_\p$ (propagating) and $u_\e$ (evanescent) satisfy the Helmholtz equation (wave number=$k$) and modified Helmholtz equation (wave number=$\text{i}k$) in $\mathbb{R}^2\setminus\overline{D},$ respectively. Notice that the auxiliary fields satisfy 
\begin{align}
     (\Delta +k^2)u_\p=0\quad \text{ and } \quad (\Delta-k^2)u_\e=0\quad\text{in}\hspace{.2cm} \mathbb{R}^2\setminus\overline{D},
     \label{hemholtzmodified} 
\end{align}
 along with the Sommerfeld radiation condition for $u_\p$ and $u_\e$, i.e.,
\begin{equation}
\lim_{r \to \infty} \sqrt{r}(\partial_r{u_\p} - \text{i} k u_\p) = 0,\quad \text{and}\quad \lim_{r \to \infty} \sqrt{r}(\partial_r{u_\e} - \text{i} k u_\e) = 0. 
\label{radiupue}
\end{equation}

The transmission eigenvalue problem arises when one asks if there exists an incident wave $u^i \neq 0$ such that the measured Cauchy data $(u^s,\partial_\nu u^s)=(0,0)$ on some curve $\Gamma=\partial \Omega$ such that $\overline{D} \subset \Omega$.  It has been shown in \cite{novelbi} that the exterior clamped obstacle problem 
$$\Delta^2 u^s-k^4 u^s=0 \quad \text{in $\mathbb{R}^2 \setminus \overline{\Omega}$} \quad \text{and} \quad (u^s,\partial_\nu u^s) \big|_{\Gamma}=(0,0) $$
along with the radiation condition \eqref{direct4} only admits a trivial solution $u^s=0$ in $\mathbb{R}^2 \setminus \overline{\Omega}$. This would imply that $u_\p = 0$ and $u_\e = 0$ 
in $\mathbb{R}^2 \setminus \overline{D}$ by unique continuation. Clearly this would imply that the scattered field $u^s=0$ in $\mathbb{R}^2 \setminus \overline{D}$. Now, due to the fact that $u^i$ is a smooth solution to \eqref{direct1}
we have the following homogeneous boundary value problem. Find the pair $(w, v) \in H^2(D) \times H^2(D)$ satisfying:
\begin{align}
\Delta^2 w - k^4 n(x) w = 0, \quad \Delta^2 v - k^4 v = 0 \quad & \text{in } D, \label{TEprob1} \\
w = v, \quad \partial_\nu w = \partial_\nu v, \quad \Delta w = \Delta v, \quad \partial_\nu \Delta w = \partial_\nu \Delta v \quad & \text{on } \partial D. \label{TEprob2}
\end{align}
This is given by the direct scattering problem, where $w = u$ the total field and $v = u^i$ the incident field. This problem arises in studying reconstruction methods like the linear sampling method \cite{IST of TA} for recovering the scatterer from the measured Cauchy data on some curve $\Gamma$.

To facilitate the analysis, we assume that $\partial D$ is of class $\mathcal{C}^4$, and $n$ is real--valued with either $n > 1$ or $0 < n < 1$ almost everywhere in $D$. The transmission eigenvalues are defined as the values of $k$ for which the above system \eqref{TEprob1}--\eqref{TEprob2} has nontrivial solutions. Furthermore, we define the Sobolev space 
\begin{align}\label{funcspace}
H^4_*(D)=\{u\in H^4(D): u= \partial_{\nu}u=\Delta u=\partial_{\nu}\Delta u=0\quad \text{on}\quad \partial D\}\subset H^4(D)
\end{align}
that is equipped with the standard $H^4(D)$ norm. Now, we introduce $u := w - v \in H^4_*(D)$ and obtain an equivalent formulation: 
\begin{align*}
\Delta^2 u - k^4 n u = -k^4(1 - n) v \quad & \text{in } D, \\
u = \partial_\nu u = \Delta u = \partial_\nu \Delta u = 0 \quad & \text{on } \partial D.
\end{align*}
This leads to the following operator equation: find $k \in \C$ and $u \in H^4_*(D)$ such that
\begin{align}
    (\Delta^2-k^4)\frac{1}{n-1}(\Delta^2u-k^4nu)=0\quad &\text{in}\quad D, \label{ueigprob1}\\
      u=\partial_\nu u=\Delta u=\partial_{\nu}\Delta u=0 \quad &\text{on}\quad \partial D.\label{ueigprob2}
\end{align}
By appealing to Green's 2nd identity, we have that the equivalent weak formulation reads: find $k \in \C$ and $u \in H^4_*(D) \setminus \{0\}$ such that
\begin{equation}
    \int_D \frac{1}{n-1}(\Delta^2u-k^4nu)(\Delta^2\overline{\varphi}-k^4\overline{\varphi})\text{d}x=0\quad \text{for all} \quad \varphi\in H^4_*(D).\label{variform}
\end{equation}
Notice, that the original eigenfunctions $w$ and $v$ can be reconstructed from $u$ via:
$$v = -\frac{1}{k^4 (1 - n)}(\Delta^2 u - k^4 n u), \quad w = -\frac{1}{k^4 (1 - n)}(\Delta^2 u - k^4 u).$$
By definition the eigenfunctions $(w,v)$ defined above from $u$ are in $H^2(D)$ for smooth $n$. 

Note that a transmission eigenvalue problem for biharmonic scattering by a clamped planar region has been studied in \cite{clamedTE}, where they used the theoretical results in \cite{te-cbc} to prove existence of transmission eigenvalues. Similar, we will employ the theory developed in \cite{te-cbc} to analyze \eqref{ueigprob1}--\eqref{ueigprob2}. The focus of this paper is to analyze this transmission eigenvalue problem and the inverse spectral theory in the context of penetrable biharmonic scattering obstacles.

Throughout the paper, we will use an equivalence of norm for the Sobolev space $H^4_*(D)$. It is well--known that by appealing to elliptic regularity (see, e.g., Evans \cite{evans}) that in the Sobolev space $H^2_0(D)$ we have that there exist constants \( C_1, C_2 > 0 \) such that
\begin{align}
C_1 \|\Delta \phi\|_{L^2(D)} \leq \|\phi\|_{H^2(D)} \leq C_2 \|\Delta \phi\|_{L^2(D)}. \label{h20-inequ}
\end{align}
We will prove a similar result for the Sobolev space $H^4_*(D)$ considered here. 

\begin{theorem}\label{norm-equ}
Let \( D \subset \mathbb{R}^2 \) be a bounded domain with \( C^4 \)–smooth boundary \( \partial D \), then there exists constants \( C_1, C_2 > 0 \) such that
$$C_1 \|\Delta^2 u\|_{L^2(D)} \leq \|u\|_{H^4(D)} \leq C_2 \|\Delta^2 u\|_{L^2(D)} \quad \text{for all} \quad u \in H^4_*(D),$$
where the Sobolev space $H^4_*(D)$ is as defined in \eqref{funcspace}.
\end{theorem}
\begin{proof}
To prove the claim, for any $u \in H^4_*(D)$ we let  $\phi = \Delta u$,  where we have that $\phi \in H^2_0(D)$. Therefore, by well--posedness of the Poisson problem with zero Dirichlet data and elliptic regularity we have that 
$$\| u \|_{H^4(D)} \leq C \| \phi \|_{H^2(D)} \quad \text{ which implies that } \quad \| u \|_{H^4(D)} \leq C \| \Delta \phi \|_{L^2(D)}$$
since $\phi \in H^2_0(D)$. Now, by definition we have that 
$$\| u \|_{H^4(D)} \leq C \| \Delta \phi \|_{L^2(D)}  = C \| \Delta^2 u \|_{L^2(D)}.$$
For the other inequality, it is clear that $\| \Delta^2 u \|_{L^2(D)} \leq  \| u \|_{H^4(D)}$, proving the claim.
\end{proof}

With this, we are now ready to analyze the transmission eigenvalue problem \eqref{ueigprob1}--\eqref{ueigprob2}. In the following section, we will prove the discreteness of the transmission eigenvalues in the complex plane by studying the variational formulation \eqref{variform}.

\section{Discreteness of Transmission Eigenvalues}\label{discrete}
In this section, we study the discreteness of the transmission eigenvalues for the biharmonic scattering problem. In general, it is known that for the analogous acoustic scattering problem that sampling methods such as the factorization method \cite{fmconductbc,kirschbook} do not provide valid reconstructions of $D$, if the wave number $k$ is a transmission eigenvalue. Similar methods to show discreteness have been applied for acoustic second-order problems (see \cite{TE2CB} for example), but here we consider the case associated with biharmonic scattering. 

We begin by assuming that the refractive index $n(x)\in L^{\infty}(D)$. Recall that we have the variational formulation  
$$\int_D \frac{1}{n-1}(\Delta^2u-k^4nu)(\Delta^2\overline{\varphi}-k^4\overline{\varphi})\text{d}x=0\quad \text{for all} \quad \varphi\in H^4_*(D).$$
By expanding the above express we have that 
$$\int_D \frac{1}{n-1} \Delta^2 u \Delta^2 \overline{\varphi} -\frac{k^4}{n-1}\Delta^2 u \overline{\varphi} - \frac{k^4n }{n-1} u \Delta^2 \overline{\varphi}+\frac{k^8 n }{n-1} u \overline{\varphi} \, \text{d}x=0$$
for all $\varphi\in H^4_*(D)$. From this, the first thing we observe is that $k=0$ is not a transmission eigenvalue provided that either 
\begin{equation}\label{assume-n}
n-1\geq n_{\text{min}}-1 >0 \quad \text{ or } \quad  1-n \geq 1-n_{\text{max}} >0 \quad \text{a.e. in the scatterer $D$}.
\end{equation}
Here $n_{\text{min}}$ and $n_{\text{max}}$ are the infimum  and supremum of $n(x)$ over $x \in D$. 
This is due to Theorem \ref{norm-equ}, which implies that 
$$\int_D \frac{1}{n-1} \Delta^2 u \Delta^2 \overline{\varphi} \, \text{d}x $$
is coercive in $H^4_*(D)$, by our assumptions in \eqref{assume-n} on $n$. Now, inspired by the analysis in \cite{fem-te} we define the auxiliary function $\psi \in H^2_0(D)$ that satisfies 
$$\Delta ^2 \psi =k^4\frac{n}{n-1}u \quad \text{in}\,\, D  \quad \text{for any}\quad u \in H^4_*(D).$$
It is clear that $\psi \in H^2_0(D)$ defined above exists and depends continuously on $u \in H^4_*(D)$.

We can rewrite transmission eigenvalue problem as the problem of finding a nontrivial pair of functions $(u,\psi)\in H^4_*(D)\times H^2_0(D)$ and a constant $k \in \mathbb{C}$ such that 
\begin{align}
\int_D \frac{1}{n-1} \Delta^2 u \Delta^2 \overline{\varphi}\, \text{d}x &= -k^4 \int_D \frac{1}{n-1}\Delta^2 u \overline{\varphi} + \frac{n }{n-1} u \Delta^2 \overline{\varphi} + \Delta \psi \Delta \overline{\varphi} \, \text{d}x \label{new-system1}\\
\int_D \Delta \psi \Delta \overline{\rho} \, \text{d}x &= k^4  \int_D \frac{n}{n-1} u \overline{\rho}  \, \text{d}x \quad \text{for all} \quad (\varphi,\rho)\in H^4_*(D)\times H^2_0(D).  \label{new-system2}
\end{align}
It is clear that \eqref{new-system1}--\eqref{new-system2} is equivalent to our original transmission eigenvalue problem. To this end, similar to \cite{fem-te} we can use \eqref{new-system1}--\eqref{new-system2} to analyze our transmission eigenvalue problem. In particular we will use this to prove the discreteness of the transmission eigenvalues. 
Furthermore, we have that adding or subtracting \eqref{new-system1}--\eqref{new-system2} gives the variational form 
\begin{align}\label{formabc}
A_{\pm}  \big( (u,\psi),(\varphi,\rho) \big) = -k^4 \left[ b\big( (u,\psi),(\varphi,\rho) \big) \mp c\big( (u,\psi),(\varphi,\rho) \big) \right]. 
\end{align}
Here, we define the sesquilinear form $A_{\pm}\big(  \cdot \, ,\cdot \big)$ given by  
\begin{align}
A_{\pm}\big( (u,\psi),(\varphi,\rho) \big) := \int_D \frac{1}{n-1} \Delta^2 u \Delta^2 \overline{\varphi}\, \text{d}x \pm \int_D \Delta \psi \Delta \overline{\rho} \, \text{d}x,  \label{forma}
\end{align}
where the sesquilinear form $b\big(  \cdot \, ,\cdot \big)$ and $c\big(  \cdot \, ,\cdot \big)$ are defined by
\begin{align}
b\big( (u,\psi),(\varphi,\rho) \big) &:= \int_D \frac{1}{n-1}\Delta^2 u \overline{\varphi} + \frac{n }{n-1} u \Delta^2 \overline{\varphi} +\Delta \psi \Delta \overline{\varphi} \, \text{d}x \label{formb}
\end{align}
and
\begin{align}
c\big( (u,\psi),(\varphi,\rho) \big) := \int_D \frac{n}{n-1} u \overline{\rho}  \, \text{d}x. \label{formc}
\end{align}
respectively, from the right-hand side in \eqref{new-system1}--\eqref{new-system2}. Notice, that this formulation of the transmission eigenvalue problem turns \eqref{variform} into a linear eigenvalue problem for $k^4$.

\begin{lemma}\label{forma-coercive}
The sesquilinear form $A_{\pm}\big(  \cdot \, ,\cdot \big)$ defined by \eqref{forma} acting on the product space $H^4_*(D)\times H^2_0(D)$ is coercive provided that \eqref{assume-n} is satisfied. 
\end{lemma} 
\begin{proof} 
This is clear from the assumptions on $n$ in \eqref{assume-n} that $A_{\pm}\big(  \cdot \, ,\cdot \big)$ is coercive by appealing to Theorem \ref{norm-equ} and equation \eqref{h20-inequ}. 
\end{proof}

The above result implies that either $A_{\pm}\big(  \cdot \, ,\cdot \big)$ is an inner--product on $H^4_*(D)\times H^2_0(D)$. We now consider the sesquilinear forms $b\big(  \cdot \, ,\cdot \big)$ defined by \eqref{formb} and $c\big(  \cdot \, ,\cdot \big)$ defined by  \eqref{formc}. In order to prove discreteness of the transmission eigenvalues, we prove that both sesquilinear forms are compact. For $c\big(  \cdot \, ,\cdot \big)$ this is a simple consequence of the Rellich--Kondrachov Embedding Theorem \cite{evans}. For the sesquilinear form $b\big(  \cdot \, ,\cdot \big)$ however this fact is harder to obtain. First, we see that by some simple calculations we have that 
\begin{align}
b\big( (u,\psi),(\varphi,\rho) \big) &= \int_D \frac{1}{n-1}\Delta^2 u \overline{\varphi} + \frac{n }{n-1} u \Delta^2 \overline{\varphi} + \Delta \psi \Delta \overline{\varphi} \, \text{d}x \nonumber \\ 
&=\int_D \frac{1}{n-1} \left[ \Delta^2 u\overline{\varphi} +  u \Delta^2 \overline{\varphi}  \right] + u \Delta^2 \overline{\varphi}  + \Delta \psi \Delta \overline{\varphi} \, \text{d}x \nonumber \\ 
&=\int_D \frac{1}{n-1} \left[ \Delta^2 u\overline{\varphi} +  u \Delta^2 \overline{\varphi}  \right] + u \Delta^2 \overline{\varphi}  + \psi \Delta^2 \overline{\varphi} \, \text{d}x, \label{formb-alt}
\end{align}
where we have used Green's second identity to obtain the last line. With this, we can now prove the compactness.

\begin{lemma}\label{forms-compact}
The sesquilinear forms $b\big(  \cdot \, ,\cdot \big)$ defined by \eqref{formb} and $c\big(  \cdot \, ,\cdot \big)$ defined by  \eqref{formc} acting on the product space $H^4_*(D)\times H^2_0(D)$ are compact provided that \eqref{assume-n} is satisfied. 
\end{lemma} 
\begin{proof}
Here, we will start with the simpler case for $c\big(  \cdot \, ,\cdot \big)$ defined by  \eqref{formc}. Therefore, we have that by the fact that $n \in L^{\infty}(D)$ and \eqref{assume-n} we have that 
$$ \left|c\big( (u,\psi),(\varphi,\rho) \big)\right| \leq C  \|u\|_{L^2(D)} \| \rho \|_{H^2(D)}.$$ 
Now, consider a sequence  $(u_j ,\psi_j) \weakc (0,0)$ as $j \to \infty$ in $H^4_*(D)\times H^2_0(D)$. This would imply that $u_j \to 0$ as $j \to \infty$ in $L^2(D)$ since $H^4(D)$ is compactly embedded in $L^2(D)$. From the above estimate we have that 
$$\sup_{\|(\varphi,\rho) \|_{H^4_*(D)\times H^2_0(D)}=1} \left| c\big( (u_j,\psi_j),(\varphi,\rho) \big) \right| \leq C \|u_j\|_{L^2(D)} \longrightarrow 0 \quad \text{as}\quad  j \to \infty$$ 
proving compactness.

Now, we consider the sesquilinear forms $b\big(  \cdot \, ,\cdot \big)$ defined by \eqref{formb}.  To this end, notice that 
$$ \left| \int_D u \Delta^2 \overline{\varphi}  + \psi \Delta^2 \overline{\varphi} \, \text{d}x \right| \leq C \Big\{ \|u\|_{L^2(D)} + \| \psi \|_{L^2(D)} \Big\}   \|\Delta^2 {\varphi}\|_{L^2(D)}. $$
This gives an estimate for the last two terms in \eqref{formb-alt}. To estimate the first two terms by appealing to the Riesz representation theorem there is a bounded linear operator $\mathbb{F}$ mapping $H^4_*(D)\times H^2_0(D)$ into itself such that 
$$\big( \mathbb{F}(u,\psi),(\varphi,\rho) \big)_{H^4_*(D)\times H^2_0(D)} := \int_D \frac{1}{n-1}  u \Delta^2 \overline{\varphi} \, \text{d}x$$ 
which implies that its adjoint satisfies
$$\big( \mathbb{F}^*(u,\psi),(\varphi,\rho) \big)_{H^4_*(D)\times H^2_0(D)} := \int_D \frac{1}{n-1}  \Delta^2 u\overline{\varphi} \, \text{d}x.$$ 
It is clear that 
$${\displaystyle \big\| \mathbb{F}(u,\psi)\big\|_{H^4_*(D)\times H^2_0(D)} \leq C \| u \|_{L^2(D)}}$$
which implies that $\mathbb{F}$ is compact by the compact embedding of $H^4_*(D)$ into $L^2(D)$. Also notice that since $\mathbb{F}$ is compact we can conclude that $\mathbb{F}^*$ is also compact. With this we have that 
\begin{align*}
\sup_{\|(\varphi,\rho) \|_{H^4_*(D)\times H^2_0(D)}=1}  \left| b\big( (u,\psi) ,(\varphi,\rho) \big) \right| &\leq C \Big\{  \|u\|_{L^2(D)} + \| \psi \|_{L^2(D)}  +\\
& \big\| \mathbb{F}(u,\psi)\big\|_{H^4_*(D)\times H^2_0(D)} + \big\| \mathbb{F}^*(u,\psi)\big\|_{H^4_*(D)\times H^2_0(D)} \Big\}.
\end{align*}
We can again, consider a sequence  $(u_j ,\psi_j) \weakc (0,0)$ as $j \to \infty$ in $H^4_*(D)\times H^2_0(D)$. This would imply that $u_j \to 0$ and $\psi_j \to 0$ as $j \to \infty$ in $L^2(D)$. We also have that 
$$\big\| \mathbb{F}(u_j,\psi_j) \big\|_{H^4_*(D)\times H^2_0(D)} + \big\| \mathbb{F}^*(u_j,\psi_j) \big\|_{H^4_*(D)\times H^2_0(D)} \longrightarrow 0 \quad \text{as}\quad  j \to \infty$$
since $\mathbb{F}$ and $\mathbb{F}^*$ is compact. This implies that 
$$\sup_{\|(\varphi,\rho) \|_{H^4_*(D)\times H^2_0(D)}=1}  \left| b\big( (u_j,\psi_j) ,(\varphi,\rho) \big) \right|\longrightarrow 0 \quad \text{as}\quad  j \to \infty$$ 
proving that $b\big(  \cdot \, ,\cdot \big)$ is compact. 
\end{proof}

Now, we can proceed with proving the discreteness of the transmission eigenvalues. Recall, that we have that $k$ is a transmission eigenvalue if and only if there is a non--trivial solution to equation \eqref{formabc}. With this we consider the corresponding source problem, i.e. where we let $k^4 u =f$ and $k^4 \psi =g$ in \eqref{formabc}. Therefore, we consider the associated source for any $(f,g) \in H^4_*(D)\times H^2_0(D)$ that is given by
 \begin{align}\label{sourceprob}
A_{\pm}  \big( (u,\psi),(\varphi,\rho) \big) = -\left[ b\big( (f,g),(\varphi,\rho) \big) \mp c\big( (f,g),(\varphi,\rho) \big) \right] 
\end{align}  
which is well--posed for $(u,\psi) \in H^4_*(D)\times H^2_0(D)$ by Lemma \ref{forma-coercive} and appealing to the Lax--Milgram Lemma. So we can define the bounded linear operator 
$$T:H^4_*(D)\times H^2_0(D) \to H^4_*(D)\times H^2_0(D) \quad \text{such that} \quad T(f,g) = (u,\psi), $$
where $(u,\psi)$ satisfy \eqref{sourceprob} which implies that 
$$A_{\pm}  \big( T(f,g),(\varphi,\rho) \big) = -\left[ b\big( (f,g),(\varphi,\rho) \big) \mp c\big( (f,g),(\varphi,\rho) \big) \right].$$ 
With this, we have that $k$ is a transmission eigenvalue if and only if $\lambda=1/k^4$ is an eigenvalue for the operator $T$ since 
$$T(u,\psi) = \lambda (u,\psi) \implies \lambda A_{\pm} \big( (u,\psi) ,(\varphi,\rho) \big) = -\left[ b\big( (u,\psi),(\varphi,\rho) \big) \mp c\big( (f,g),(\varphi,\rho) \big) \right].$$  
Using the operator $T$ we can now prove the discreteness of the transmission eigenvalues. 

\begin{theorem}\label{TEdiscrete}
Assume that $n \in L^{\infty}(D)$ satisfies \eqref{assume-n}, then the set of transmission eigenvalues is discrete with no finite accumulation point.
\end{theorem}
\begin{proof}
Notice that $k$ is a transmission eigenvalue implies that $1/k^4$ is in the spectrum of the operator $T$. To prove the claim, we first show that $T$ is compact. To this end, by the definition of the sesquilinear form $A_{\pm}\big(  \cdot \, ,\cdot \big)$ and Lemma \ref{forma-coercive} we have that $A_{\pm}\big(  \cdot \, ,\cdot \big)$ is an inner--product on $H^4_*(D)\times H^2_0(D)$. 
This implies that 
\begin{align*}
\|T(f,g)\|_{H^4_*(D)\times H^2_0(D)} &= \sup_{\|(\varphi,\rho) \|_{H^4_*(D)\times H^2_0(D)}=1} \left|A_{\pm}\big(T(f,g) ,(\varphi,\rho)\big)\right| \\
&=\sup_{\|(\varphi,\rho) \|_{H^4_*(D)\times H^2_0(D)}=1} \left| b\big( (f,g) ,(\varphi,\rho)\big) \mp c\big( (f,g) ,(\varphi,\rho)\big) \right|.
\end{align*}
With this, we see that for any sequence $(f_j ,g_j) \weakc (0,0)$ as $j \to \infty$ in $H^4_*(D)\times H^2_0(D)$ we have that 
$$ \|T(f_j ,g_j)\|_{H^4_*(D)\times H^2_0(D)} 
=\sup_{\|(\varphi,\rho) \|_{H^4_*(D)\times H^2_0(D)}=1} \left| b\big( (f_j ,g_j) ,(\varphi,\rho)\big) \mp c\big( (f_j ,g_j) ,(\varphi,\rho)\big) \right|$$ 
and since the sesquilinear forms $b\big(  \cdot \, ,\cdot \big)$ and $c\big(  \cdot \, ,\cdot \big)$ are compact by Lemma \ref{forms-compact} we have that 
$$\|T(f_j ,g_j)\|_{H^4_*(D)\times H^2_0(D)} \longrightarrow 0 \quad \text{as}\quad  j \to \infty$$
proving that $T$ is compact. Now, since $T$ is compact we have that its spectrum is discrete with only $0$ as the only accumulation point, proving the claim. 
\end{proof}

From Theorem \ref{TEdiscrete} we are able to prove discreteness of the transmission eigenvalue by considering an equivalent eigenvalue problem for a compact operator. Notice, that since the sesquilinear form $b\big(  \cdot \, ,\cdot \big)$ is not self--adjoint, we can not conclude existence of the transmission eigenvalues with this argument. This could also imply the existence of complex eigenvalues since the associated operator may not be self--adjoint.


\section{Existence of Transmission Eigenvalues}\label{exist}
\noindent In this section, we establish the existence of transmission eigenvalues for the biharmonic scattering problem previously presented. Our objective is to demonstrate that real transmission eigenvalues exist, following an approach inspired by \cite{te-cbc}. We work with the eigenvalue problem \eqref{ueigprob1}--\eqref{ueigprob2} and its variational formulation \eqref{variform}:  
\begin{equation*}
    \int_D \frac{1}{n-1}\big(\Delta^2u-k^4nu\big)\big(\Delta^2\overline{\varphi}-k^4\overline{\varphi}\big)\,\text{d}x=0
    \quad \forall\varphi\in H^4_*(D).
\end{equation*}
We will leverage the following two conditions on our operators to show the existence of transmission eigenvalues for our scattering problem (see Theorem 2.3 in \cite{cakonihaddar} for more details). 
\begin{theorem}\label{cakonihaddar}
    We first defined a continuous mapping $k \mapsto A_k$ from $(0,\infty)$ into the set of self--adjoint, positive definite, bounded operators on a Hilbert space $U$, and note that $B$ is a self--adjoint, non--negative, compact operator on $U$. Assume there exist $k_0,k_1>0$ such that  
\begin{enumerate}\label{cakonihaddar}
    \item $A_{k_0}-k_0^4B$ is positive definite,  
    \item $A_{k_1}-k_1^4B$ is non--positive on an $m$--dimensional subspace of $U$.  
 \end{enumerate}
 Then, for each $j=1,\dots,m$, there exists $k \in [k_0,k_1]$ such that $\lambda_j(k)=k^4$, where $\lambda_j(k)$ is the $j$-th transmission eigenvalue of $A_k-\lambda B$.
\end{theorem}
We return to the variational form of \eqref{ueigprob1}--\eqref{ueigprob2}, and introduce three sesquilinear forms acting on  $H^4_*(D) \times H^4_*(D)$:  
\begin{align}
\mathcal{A}_k(u,\varphi)&=\int_D \frac{1}{n-1}(\Delta^2 u-k^4u)(\Delta^2 \overline{\varphi}-k^4\overline{\varphi})+k^8u\overline{\varphi}\,\text{d}x, \label{A}\\
\Tilde{\mathcal{A}}_k(u,\varphi)&=\int_D \frac{n}{1-n}(\Delta^2 u-k^4u)(\Delta^2 \overline{\varphi}-k^4\overline{\varphi})+\Delta^2 u\,\Delta^2 \overline{\varphi}\,\text{d}x, \label{Ati}
\end{align}
and
\begin{align}
\mathcal{B}(u,\varphi)&=\int_D \Delta u\,\Delta \overline{\varphi}\,\text{d}x. \label{B}
\end{align}

The eigenvalue problem can then be expressed in two equivalent ways, depending on $n$:  
\begin{equation}
    \mathcal{A}_k(u,\varphi)-k^4\mathcal{B}(u,\varphi)=0
    \quad \forall \varphi\in H^4_*(D), \qquad n>1,
\end{equation}
or
\begin{equation}
    \Tilde{\mathcal{A}}_k(u,\varphi)-k^4\mathcal{B}(u,\varphi)=0
    \quad \forall \varphi\in H^4_*(D), \qquad 0<n<1.
\end{equation}

Now, by the Riesz representation theorem, we associate bounded operators $\mathbb{A}_k, \Tilde{\mathbb{A}}_k,\mathbb{B}:H^4_*(D)\to H^4_*(D)$ to the sesqulinear forms stated above, where they are defined by  
$$(\mathbb{A}_k u, \varphi)_{H^4(D)}=\mathcal{A}_k(u,\varphi), \quad (\Tilde{\mathbb{A}}_k u, \varphi)_{H^4(D)}=\Tilde{\mathcal{A}}_k(u,\varphi), \quad (\mathbb{B} u, \varphi)_{H^4(D)} = \mathcal{B}(u,\varphi).$$

Because $n$ is real--valued, the sesquilinear forms \eqref{A}--\eqref{B} are Hermitian, hence the induced operators are self--adjoint. Moreover, the compact embedding $H^4(D)\hookrightarrow H^2(D)$ ensures that $\mathbb{B}$ is compact, and positivity follows from the standard well--posedness. We now switch our attention to the operators $\mathcal{A}(\cdot,\cdot)$ and $\Tilde{\mathcal{A}}(\cdot,\cdot)$.
\begin{theorem}
    The operators $\mathcal{A}(\cdot,\cdot)$ and $\Tilde{\mathcal{A}}(\cdot,\cdot)$ are coercive for the case $n>1$ and $0<n<1,$ respectively.
\end{theorem}
\begin{proof}
    For $n>1$,  
\begin{align*}
    \mathcal{A}_k(u,u)&=\int_D \frac{1}{n-1}\big|\Delta^2 u-k^4u\big|^2 \,\text{d}x + k^8\|u\|^2_{L^2(D)}\\
    &\geq \alpha \|\Delta^2u -k^4u\|^2_{L^2(D)} +k^8\|u\|^2_{L^2(D)}, \qquad \alpha = \tfrac{1}{n_{\max}-1}.
\end{align*}
After expanding and applying Young’s inequality, one obtains  
\begin{align*}
    \mathcal{A}_k(u,u)&\geq \alpha \|\Delta^2u\|^2_{L^2(D)}-2k^4 \alpha \|u\|_{L^2(D)}\|\Delta^2u\|_{L^2(D)}+k^8(\alpha+1)\|u\|^2_{L^2(D)} \\
    &\geq \left(\alpha -{\alpha^2}/{\epsilon} \right)\|\Delta^2 u \|^2_{L^2(D)}+k^8(\alpha+1-\epsilon)\|u\|^2_{L^2(D)}
\end{align*}
for any $\alpha < \epsilon < \alpha+1$. If we pick $\epsilon = \alpha+1/2$ we have that 
$$\mathcal{A}_k(u,u)\geq \frac{\alpha}{1+2\alpha} \|\Delta^2u\|^2_{L^2(D)}$$
which implies coercivity since ${\alpha}/{(1+2\alpha)}>0$ and Theorem \ref{norm-equ}. 
For $\Tilde{\mathcal{A}}_k$ in the case $0<n<1,$ we have 
\begin{align*}
    \Tilde{\mathcal{A}}_k(u,u)&=\int_D \frac{n}{1-n}|\Delta^2 u-k^4u|^2 \, \text{d}x+\|\Delta^2u\|^2_{L^2(D)}\\
    &\geq \|\Delta^2 u \|^2_{L^2(D)}
\end{align*}
which again proves coercivity by Theorem \ref{norm-equ}. Therefore, in either case we have that $\mathcal{A}(\cdot,\cdot)$ for $n>1$ and $\Tilde{\mathcal{A}}(\cdot,\cdot)$ for $0<n<1$ there are constants independent of $k$ such that 
$$ \mathcal{A}_k(u,u)\geq C_1 \|u\|^2_{H^4(D)}\quad  \Tilde{\mathcal{A}}_k(u,u)\geq C_2 \|u\|^2_{H^4(D)}$$
where the constants $C_1$ and $C_2$ are positive for all $u\in H^4_*(D)$.
\end{proof}
From Theorem \eqref{cakonihaddar}, it becomes necessary to specify additional properties of the operators $\mathbb{A}_k$, $\Tilde{\mathbb{A}}_k$, and $\mathbb{B}$. The following theorem provides the required assumptions.
\begin{corollary}\label{operators}
    Suppose $n>1$ or $0<n<1$ almost everywhere in $D$. Then:  
    \begin{enumerate}
        \item $\mathbb{B}$ is compact, positive, and self--adjoint,  
        \item $\mathbb{A}_k$ is a coercive, self--adjoint operator when $n>1$,  
        \item $\Tilde{\mathbb{A}}_k$ is coercive and self--adjoint when $0<n<1$.  
    \end{enumerate}
    Consequently, the operators $\mathbb{A}_k-k^4\mathbb{B}$ and $\Tilde{\mathbb{A}}_k-k^4\mathbb{B}$ will satisfy the Fredholm property.
\end{corollary}
We are able to note that the transmission eigenvalues arise as solutions of the generalized eigenvalue problem that depends on the operators 
\begin{equation}
    \mathbb{A}_ku=\lambda_j(k)\mathbb{B}u \quad (n>1),
    \qquad \Tilde{\mathbb{A}}_ku=\lambda_j(k)\mathbb{B}u \quad (0<n<1), \label{operatorsAB}
\end{equation}
with $\lambda_j(k)-k^4=0$. By the above operator properties, we have almost satisfied the aforementioned Theorem 2.3 in \cite{cakonihaddar}. It remains to verify that the operators are positive for some $k_0$ and non-positive for some $k_1$ on a finite dimensional subspace of $H_*^4(D).$ The following result addresses what is left to show for the operators. 

\begin{theorem}\label{thm4.2}
    Assume $n>1$ or $0<n<1$ a.e. in $D$. Then, for sufficiently small $k>0$, one has
    \[
    \mathcal{A}_k(u,u)-k^4\mathcal{B}(u,u)\geq\delta\|u\|^2_{H^4(D)}
    \quad \text{or}\quad
    \Tilde{\mathcal{A}}_k(u,u)-k^4\mathcal{B}(u,u)\geq\delta\|u\|^2_{H^4(D)},
    \]
    for all $u\in H^4_*(D)$ and some $\delta>0$.
\end{theorem}

\begin{proof}
    We first take address the case $0<n<1$, so  
    \begin{align*}
        \Tilde{\mathcal{A}}_k(u,u)-k^4\mathcal{B}(u,u)&=\int_D\frac{n}{1-n}|\Delta^2 u-k^4u|^2+|\Delta^2 u|^2\text{d}x-k^4\int_D|\Delta u|^2\text{d}x\\
        & \geq \|\Delta^2u\|^2_{L^2(D)}-k^4\|\Delta u\|^2_{L^2(D)}.
    \end{align*}
    Using Theorem \ref{norm-equ}, recall that we have that there exists $C_1>0$ such that 
    $$\|\Delta u\|^2_{L^2(D)}\leq \| u \|_{H^4(D)} \leq C_1\|\Delta^2 u\|^2_{L^2(D)}.$$ 
    Using this estimate gives us that 
    \begin{align*}
        \Tilde{\mathcal{A}}_k(u,u)-k^4\mathcal{B}(u,u)&\geq \|\Delta^2u\|^2_{L^2(D)}-C_1k^4|\Delta^2u\|^2_{L^2(D)}\\
        &=(1- k^4C_1)\|\Delta^2u\|^2_{L^2(D)}.  
    \end{align*}
    Thus, for all $k>0$ sufficiently small, we have that 
    $$\Tilde{\mathcal{A}}_k(u,u)-k^4\mathcal{B}(u,u)\geq\delta\|u\|^2_{H^4(D)} \quad \text{for some $\delta >0$.}$$ 
    
    We switch our attention to $n>1,$ where we have 
    \begin{align*}
        \mathcal{A}_k(u,u)-k^4\mathcal{B}(u,u)&=\int_D\frac{1}{n-1}|\Delta^2 u-k^4u|^2+k^8| u|^2\text{d}x-k^4\int_D|\Delta u|^2\text{d}x\\
       & \geq C\|\Delta^2 u\|^2_{L^2(D)}-k^4\|\Delta u\|^2_{L^2(D)}\\
       &\geq  C\|\Delta^2 u\|^2_{L^2(D)}-k^4C_1\| \Delta^2 u\|^2_{L^2(D)}\\
       & \geq (C-k^4C_1)\|\Delta^2 u\|^2_{L^2(D)},
    \end{align*}
    where $C_1>0$ is the constant, where $\|\Delta u\|^2_{L^2(D)}\leq C_1\|\Delta^2 u\|^2_{L^2(D)}$ for all $u\in H^4_*(D)$ and the constant $C>0$ is the constant, where 
    $$\mathcal{A}_k(u,u)= \int_D\frac{1}{n-1}|\Delta^2 u-k^4u|^2+k^8| u|^2\text{d}x\geq C\|\Delta u\|^2_{L^2(D)}\quad \text{for all}\quad u\in H^4_*(D).$$
  After this, we see that we have proven the claim as 
  $$\mathcal{A}_k(u,u)-k^4\mathcal{B}(u,u)\geq\delta\|u\|^2_{H^4(D)}\quad \text{for some $\delta >0$.}$$ 
  Now, we show the second part of Lemma \eqref{cakonihaddar}. 
\end{proof}
\begin{theorem}
    Assume $n>1$ or $0<n<1$ a.e. in $D$. Then, for some positive $k>0$,
    \[
    \mathcal{A}_k(u,u)-k^4\mathcal{B}(u,u)
    \quad \text{or}\quad
    \Tilde{\mathcal{A}}_k(u,u)-k^4\mathcal{B}(u,u),
    \]
    is non-positive on a $N$--dimensional subspace for any $N \in \N$.
\end{theorem}
\begin{proof}
    To prove the claim, we consider the case for $n>1$ and note that the other case $0<n<1$ can be proved in a similar way. To this end, we begin by defining $$B_j=B(x_j,\varepsilon)=\{x\in \mathbb{R}^{2}:|x-x_j| <\varepsilon \},$$ where $x_j\in D$ and $\varepsilon>0.$ Now, let ${M}(\varepsilon)$ be the number of disjoint balls, i.e., where $\overline{B}_j\cap\overline{B}_i$ are empty with $\varepsilon$ small enough such that $\overline{B}_j\subset D$. Since we have assumed that $n>1$ this would imply that $n_{\text{min}} >1$ and by separation of variables (see Section \ref{numerics}) we have that there exists smooth non--trivial solutions to 
\begin{align}
\Delta^2 w_j-k^4 n_{\text{min}} w_j = 0 \quad \text{ and } \quad  \Delta^2 v_j-k^4v_j = 0  \quad & \text{in} \quad   B_j, \label{transmi1}\\
w_j=v_j,\quad \partial_{\nu}w_j=\partial_{\nu}v_j , \quad \Delta w_j =\Delta v_j, \quad \text{and}\quad \partial_{\nu}\Delta w_j=\partial_{\nu}\Delta v_j  \quad & \text{on} \quad  \partial B_j .\label{transmi2}
\end{align}
In order to construct our subspace, let $u_j=v_j-w_j\in H^4_*(B_j)$. 

Now, due to the zero boundary conditions in $H^4_*(B_j)$ we have that $\Tilde{u}_j$ being the extension of $u_j$ by zero to $D$ in $H^4_*(D)$. It is clear that the supports of $\Tilde{u}_j$ are disjoint so $\{\Tilde{u}_j \, : \, j=1 , \cdots , M(\varepsilon) \}$ is an orthogonal set in $H^4_*(D).$ Then, we have that $$W_{M(\varepsilon)}=\text{span}\{ \Tilde{u}_j \, : \, j=1 , \cdots , M(\varepsilon)\}$$ 
forms an $M(\varepsilon)$ dimensional subspace of $H^4_*(D).$ Moreover, for any transmission eigenvalue $k_\varepsilon$ of \eqref{transmi1}--\eqref{transmi2} we have 
\begin{align*}
    0&=\int_D \frac{1}{n_{\text{min}} -1}(\Delta^2 \Tilde{u}_j-k^4\Tilde{u}_j)(\Delta^2 \overline{\Tilde{u}_j}-k^4 n_{\text{min}}\overline{\Tilde{u}_j})\text{d}x\\
   &= \int_{B_j} \frac{1}{n_{\text{min}} -1}(\Delta^2 \Tilde{u}_j-k^4\Tilde{u}_j)(\Delta^2 \overline{\Tilde{u}_j}-k^4 n_{\text{min}} \overline{\Tilde{u}_j})\text{d}x\\
   &=\int_{B_j} \frac{1}{n_{\text{min}} -1}|\Delta^2 \Tilde{u}_j-k^4\Tilde{u}_j|^2+k^8|\Tilde{u}_j|^2\text{d}x-k^4\int_{B_j}|\Delta \Tilde{u}_j|^2\text{d}x.
\end{align*}
Now, let $k_{\varepsilon}$ be the first transmission eigenvalue of \eqref{transmi1}--\eqref{transmi2} in some ball $B_j$ with the eigenfunction $u_j$. Notice that since \eqref{transmi1}--\eqref{transmi2} has constant coefficients that each $B_j$ has the same set of eigenvalues associated with \eqref{transmi1}--\eqref{transmi2}. Then, for the extension $\Tilde{u}_j$ we have 
\begin{align*}
    \mathcal{A}_{k_{\varepsilon}}(\Tilde{u}_j,\Tilde{u}_j)-k_\varepsilon^4\mathcal{B}_{k_{\varepsilon}}(\Tilde{u}_j,\Tilde{u}_j)&=\int_{D} \frac{1}{n -1}|\Delta^2 \Tilde{u}_j-k_{\varepsilon}^4\Tilde{u}_j|^2+k_{\varepsilon}^8|\Tilde{u}_j|^2\text{d}x-k_{\varepsilon}^4\int_{B_j}|\Delta \Tilde{u}_j|^2\text{d}x\\
    & \leq \int_{B_j} \frac{1}{n_{\text{min}} -1}|\Delta^2 \Tilde{u}_j-k_{\varepsilon}^4\Tilde{u}_j|^2+k_{\varepsilon}^8|\Tilde{u}_j|^2\text{d}x-k_{\varepsilon}^4\int_{B_j}|\Delta \Tilde{u}_j|^2\text{d}x=0.
\end{align*}
Therefore, by appealing to linearity and the disjoint supports we obtain that 
$$\mathcal{A}_{k_{\varepsilon}}({u},{u})-k_\varepsilon^4\mathcal{B}_{k_{\varepsilon}}({u},{u})\leq 0 \quad \text{ for all $u\in W_{M(\varepsilon)}$}.$$ 
Note that as $\varepsilon \to 0$, we have that $M(\varepsilon)\to \infty$ which proves the claim. 
\end{proof}
\begin{theorem}
    If $n>1$ or $0<n<1$ a.e. in $D$, then the biharmonic transmission problem possesses infinitely many real transmission eigenvalues.
\end{theorem}
\begin{proof}
    The proof follows directly by applying Theorems \ref{cakonihaddar} and Corollary \ref{operators}, where we have proven that our operator satisfies the assumptions in the previous results.
\end{proof}

Thus, the biharmonic scattering problem dealing with transmission eigenvalues  admits infinitely many real transmission eigenvalues. While this has long been known in the acoustic setting, here we extended the result to biharmonic scattering. In the next section, we will examine how these eigenvalues depend on the physical parameters of the model.

\section{Monotonicity of the Transmission Eigenvalues}\label{limit}

In this section we investigate how the transmission eigenvalues depend on the material coefficient $n$. Our focus is on the behavior of the first transmission eigenvalue, and we establish that it can be seen as a monotonic functions with respect to the parameter $n$. This observation provides a means to extract information about $n$ from the spectral data. In particular, we find that:
\begin{enumerate}
    \item for $0<n<1$, the first eigenvalue grows as $n$ increases, 
    \item for $n>1$, the first eigenvalue decreases as $n$ increases.
\end{enumerate}
As a consequence, one can uniquely determine a constant refractive index from knowledge of the first transmission eigenvalue. Recall that transmission eigenvalues are roots of
\begin{equation}
    \lambda_{j}(k,n)-k^4(n)=0, \label{monotrans1}
\end{equation}
and the smallest root corresponds to the first transmission eigenvalue. More precisely, for $u\neq 0$ we have the characterizations
\begin{equation}
     \lambda_1(k;n)=\underset{u\in H^4_*(D)}{\min}\frac{\mathcal{A}_k(u,u)}{\mathcal{B}(u,u)} \quad (n>1),
     \label{min1}
 \end{equation}
and
 \begin{equation}
     \lambda_1(k;n)=\underset{u\in H^4_*(D)}{\min}\frac{\Tilde{\mathcal{A}}_k(u,u)}{\mathcal{B}(u,u)} \quad (0<n<1),
     \label{min2}
 \end{equation}
where the forms $\mathcal{A},\Tilde{\mathcal{A}}$, and $\mathcal{B}$ are defined in \eqref{A}--\eqref{B}.  
The minimizers of \eqref{min1} and \eqref{min2} are precisely the eigenfunctions associated with $\lambda_1(k;n)$. We denote the first transmission eigenvalue by $k_1(n)$ and state the main monotonicity result.

\begin{theorem}
    Let $n_1,n_2$ be constants with $1<n_1\leq n_2$ or $0<n_1\leq n_2<1$. Then the first transmission eigenvalue $k_1(n)$ satisfies
    \begin{enumerate}
        \item if $n_1>1$, then $k_1(n_2)\leq k_1(n_1)$,
        \item if $n_2<1$, then $k_1(n_1)\leq k_1(n_2)$.
    \end{enumerate}
    Furthermore, if the inequalities on $n_1,n_2$ are strict, then $k_1(n)$ is strictly monotone with respect to $n$. \label{monotrans}
\end{theorem}

\begin{proof}
    We begin with the case $n_1>1$. Let $k_i = k_1(n_i)$ denote the first transmission eigenvalue corresponding to $n=n_i$ for $i=1,2$. For any $u\in H^4_*(D)$ normalized by $\|\Delta u\|_{L^2(D)}=1$, the ordering $n_1\leq n_2$ implies
    \begin{align*}
        \lambda_1(k_1;n_2)&\leq \int_D \frac{1}{n_2-1}\big|\Delta^2 u-k_1^4u\big|^2+k_1^8|u|^2\,\text{d}x \\
        &\leq \int_D \frac{1}{n_1-1}\big|\Delta^2 u-k_1^4u\big|^2+k_1^8|u|^2\,\text{d}x.
    \end{align*}
    Choosing $u=u_1$, the normalized eigenfunction associated with $k_1(n_1)$, we obtain
    \[
      \lambda_1(k_1;n_1)= \int_D \frac{1}{n_1-1}\big|\Delta^2 u_1-k_1^4u_1\big|^2+k_1^8|u_1|^2\,\text{d}x,
    \]
    where $u_1$ is the minimizer of \eqref{min1}. Thus, we have that  $$\lambda_1(k_1;n_2)\leq \lambda_1(k_1;n_1) \quad \text{which implies that } \quad \lambda_1(k_1;n_2) - k_1^4 \leq \lambda_1(k_1;n_1)- k_1^4 = 0.$$ 
    Recall, that from Theorem \ref{thm4.2} we have that there is a $\tau>0$ sufficiently small such that 
    $$\lambda_1(\tau;n_2) - \tau^4 >0.$$
   By continuity, one has that $\lambda_1(k,n_2)-k^4$ has one root in $[\sqrt[4]{\tau},k_1]$. Since $k_1$ is the smallest root of $\lambda_1(k;n_2)-k^4$, we can conclude that $k_2\leq k_1$ which proves the claim for this case. 
    
    For the case $0<n_1\leq n_2<1$, we see that this can be handled in a similar manner. Indeed, for any $u\in H^4_*(D)$ normalized by $\|\Delta u\|_{L^2(D)}=1$, the ordering $n_1\leq n_2$ implies
    \begin{align*}
        \lambda_1(k_2;n_1)&\leq \int_D \frac{n_1}{1-n_1}|\Delta^2 u-k_2^4u|^2 \, \text{d}x+\|\Delta^2u\|^2_{L^2(D)} \\
        &\leq \int_D \frac{n_2}{1-n_2}|\Delta^2 u-k_2^4u|^2 \, \text{d}x+\|\Delta^2u\|^2_{L^2(D)}.
    \end{align*}
    Choosing $u=u_2$, the normalized eigenfunction associated with $k_1(n_2)$, we obtain
    \[
    \lambda_1(k_2;n_2)= \int_D \frac{n_2}{1-n_2}|\Delta^2 u_2-k_2^4u_2|^2 \, \text{d}x+\|\Delta^2u\|^2_{L^2(D)},
    \]
    where we see that $\lambda_1(k_2;n_1) \leq  \lambda_1(k_2;n_2)$. Therefore, the proof for this case is now similar to the previous case. 
\end{proof}
Using the proof of the previous results, we obtain a uniqueness result for the refractive index $n.$ This comes from the fact that we know that the refractive index is strictly monotone with respect to the first transmission eigenvalue. 

\begin{corollary}
    Suppose $n$ is a constant refractive index with either $n>1$ or $0<n<1$. Then $n$ is uniquely determined by the first transmission eigenvalue.
\end{corollary}

\section{Numerical Validation}\label{numerics}
In this section, we provide some numerical examples that validate the theoretical results of the previous sections. We will provide some examples for the monotonicity of the eigenvalues with respect to the given constant parameter $n$ given in Theorem \ref{monotrans} in different shapes and regions with two different methods; separation of variables and boundary integral equations. 
Recall that we have to solve the problem
\begin{align}
\Delta^2 w - k^4 n w = 0, \quad & \text{in } D, \label{eq1} \\
\Delta^2 v - k^4 v = 0, \quad & \text{in } D, \label{eq2} \\
w = v, \quad & \text{on } \partial D, \label{bc1}\\
\partial_\nu w = \partial_\nu v, \quad & \text{on } \partial D, \label{bc2}\\
\Delta w = \Delta v,  \quad & \text{on } \partial D, \label{bc3}\\
\partial_\nu \Delta w = \partial_\nu \Delta v, \quad & \text{on } \partial D. \label{bc4}
\end{align}
We use a similar decomposition to the one used for the scattered field in the exterior of the scatterer discussed in Section \ref{theproblem}. Therefore, we see that the eigenfunctions in (\ref{eq1}) and (\ref{eq2}) can be written as $w=w_{pr}+w_{ev}$ and $v=v_{pr}+v_{ev}$, that satisfy  
\[\left(\Delta+k^2\sqrt{n}\right)w_{pr}=0\quad \text{and}\quad \left(\Delta-k^2\sqrt{n}\right)w_{ev}=0,\quad \text{in } D\]
and
\[\left(\Delta+k^2\right)v_{pr}=0\quad \text{and}\quad \left(\Delta-k^2\right)v_{ev}=0,\quad \text{in } D,\]
respectively. 
\subsection{Separation of variables}
First, we explain how to compute the transmission eigenvalues for a disk of radius one.
Because of the Helmholtz decomposition for $w=w_{pr}+w_{ev}$ and $v=v_{pr}+v_{ev}$, we use the Fourier-Bessel ansatz
\begin{eqnarray*}w&=&\sum_{p\in \mathbb{N}_0}a_p^{(w)} J_p(\sqrt[4]{n}k r)\mathrm{e}^{\mathrm{i}p\theta}+b_p^{(w)} J_p(\sqrt[4]{n}\mathrm{i}k r)\mathrm{e}^{\mathrm{i}p\theta}\,,\\
v&=&\sum_{p\in \mathbb{N}_0}a_p^{(v)} J_p(k r)\mathrm{e}^{\mathrm{i}p\theta}+b_p^{(v)}J_p(\mathrm{i}k r)\mathrm{e}^{\mathrm{i}p\theta}\,.
\end{eqnarray*}
Using the first boundary condition (\ref{bc1}) yields
\[J_p(\sqrt[4]{n}k )a_p^{(w)} + J_p(\sqrt[4]{n}\mathrm{i}k )b_p^{(w)} -J_p(k )a_p^{(v)} - J_p(\mathrm{i}k )b_p^{(v)}=0\,.\]
With the second boundary condition (\ref{bc2}) gives
\[\sqrt[4]{n}J_p'(\sqrt[4]{n}k )a_p^{(w)} + \sqrt[4]{n}\mathrm{i} J_p'(\sqrt[4]{n}\mathrm{i}k )b_p^{(w)} -J_p'(k )a_p^{(v)} - \mathrm{i}J_p'(\mathrm{i}k )b_p^{(v)}=0\,,\]
where we divided by $k$. Using the third boundary condition (\ref{bc3}) yields
\[-\sqrt{n}J_p(\sqrt[4]{n}k )a_p^{(w)} +\sqrt{n} J_p(\sqrt[4]{n}\mathrm{i}k )b_p^{(w)} +J_p(k )a_p^{(v)} - J_p(\mathrm{i}k )b_p^{(v)}=0\,,\]
where we divided by $k^2$. With boundary condition (\ref{bc4}) we get
\[-\sqrt[4]{n^3}J_p'(\sqrt[4]{n}k )a_p^{(w)} +\sqrt[4]{n^3}\mathrm{i} J_p'(\sqrt[4]{n}\mathrm{i}k )b_p^{(w)} +J_p'(k )a_p^{(v)} - \mathrm{i}J_p'(\mathrm{i}k )b_p^{(v)}=0\,,\]
where we divided by $k^3$. The last four equations can be written as
\[\begin{pmatrix}
    J_p(\sqrt[4]{n}k )& J_p(\sqrt[4]{n}\mathrm{i}k )& -J_p(k )&- J_p(\mathrm{i}k )\\ 
    \sqrt[4]{n}J_p'(\sqrt[4]{n}k )& \sqrt[4]{n}\mathrm{i} J_p'(\sqrt[4]{n}\mathrm{i}k )& -J_p'(k )&-\mathrm{i} J_p'(\mathrm{i}k )\\
   -\sqrt{n}J_p(\sqrt[4]{n}k ) & \sqrt{n}J_p'(\sqrt[4]{n}\mathrm{i}k )&J_p(k ) &- J_p(\mathrm{i}k )\\
  -\sqrt[4]{n^3}J_p'(\sqrt[4]{n}k )  & \sqrt[4]{n^3}\mathrm{i} J_p'(\sqrt[4]{n}\mathrm{i}k )& J_p'(k )&- \mathrm{i}J_p'(\mathrm{i}k )
\end{pmatrix}\begin{pmatrix}
a_p^{(w)}\\
b_p^{(w)}\\
a_p^{(v)}\\
b_p^{(v)}
\end{pmatrix}=\begin{pmatrix}
0\\
0\\
0\\
0
\end{pmatrix}
\,.\]
Then, the roots of the determinant are numerically computed to high accuracy for a fixed $p\in\mathbb{N}_0$ and a given refraction index $n$.
\subsection{Boundary integral equations}

To solve the four `Helmholtz' equations associated with the eigenvalue system, we will use a single-layer ansatz
\begin{eqnarray*}
w_{pr}(x)&=&\mathrm{SL}_{k\sqrt[4]{n}}[\varphi_w](x)\,,\quad x\in D\,,\\
w_{ev}(x)&=&\mathrm{SL}_{\mathrm{i}k\sqrt[4]{n}}[\psi_w](x)\,,\quad x\in D\,,\\
v_{pr}(x)&=&\mathrm{SL}_{k}[\varphi_v](x)\,,\quad x\in D\,,\\
v_{ev}(x)&=&\mathrm{SL}_{\mathrm{i}k}[\psi_v](x),\quad x\in D,
\end{eqnarray*}
where the density functions $\varphi_w$, $\psi_w$, $\varphi_v$, and $\psi_v$ are unknown. Here, $\mathrm{SL}_{\tau}[\phi](x)$ denotes the single-layer operator with wave number $\tau$ applied to the function $\phi$ defined for $x\notin \partial D$. Using the boundary condition (\ref{bc1}) along with the jump conditions yields
\begin{eqnarray}
\mathrm{S}_{k\sqrt[4]{n}}[\varphi_w](x)+\mathrm{S}_{\mathrm{i}k\sqrt[4]{n}}[\psi_w](x)-\mathrm{S}_{k}[\varphi_v](x)-\mathrm{S}_{\mathrm{i}k}[\psi_v](x)=0,\quad x\in \partial D,
\label{equation1}
\end{eqnarray}
where $\mathrm{S}_{\tau}[\phi](x)$ denotes the single-layer operator with wave number $\tau$ applied to the function $\phi$ defined for $x\in \partial D$. Likewise, we use boundary condition (\ref{bc2}) along with the jump condition to obtain
\begin{eqnarray}
\mathrm{D}^\top_{k\sqrt[4]{n}}[\varphi_w](x)+\frac{1}{2}\varphi_w(x)+\mathrm{D}^\top_{\mathrm{i}k\sqrt[4]{n}}[\psi_w](x)+\frac{1}{2}\psi_w(x)-\mathrm{D}^\top_{k}[\varphi_v](x)\notag\\
-\frac{1}{2}\varphi_v(x)-\mathrm{D}^\top_{\mathrm{i}k}[\psi_v](x)-\frac{1}{2}\psi_v(x)=0,\quad x\in \partial D,
\label{equation2}
\end{eqnarray}
where $\mathrm{D}^\top_{\tau}[\phi](x)$ denotes the normal derivative of the single-layer operator with wave number $\tau$ applied to the function $\phi$ defined for $x\in \partial D$. Observe that we have 
\begin{eqnarray*}\Delta w&=&\Delta w_{pr}+\Delta w_{ev}=\Delta \mathrm{SL}_{k\sqrt[4]{n}}[\varphi_w]+\Delta \mathrm{SL}_{\mathrm{i}k\sqrt[4]{n}}[\psi_w]\\
&=&-k^2\sqrt{n}\mathrm{SL}_{k\sqrt[4]{n}}[\varphi_w]+k^2\sqrt{n}\mathrm{SL}_{\mathrm{i}k\sqrt[4]{n}}[\psi_w].
\end{eqnarray*}
Likewise, we obtain
\begin{eqnarray*}\Delta v=\Delta v_{pr}+\Delta v_{ev}=\Delta \mathrm{SL}_{k}[\varphi_v]+\Delta \mathrm{SL}_{\mathrm{i}k}[\psi_v]=-k^2\mathrm{SL}_{k}[\varphi_v]+k^2\mathrm{SL}_{\mathrm{i}k}[\psi_v].
\end{eqnarray*}
Using the boundary condition (\ref{bc3}) along with the jump conditions gives
\begin{eqnarray}
\sqrt{n}\mathrm{S}_{k\sqrt[4]{n}}[\varphi_w]
-\sqrt{n}\mathrm{S}_{\mathrm{i}k\sqrt[4]{n}}[\psi_w]-\mathrm{S}_{k}[\varphi_v]+\mathrm{S}_{\mathrm{i}k}[\psi_v]=0,
\label{equation3}
\end{eqnarray}
where we divide by $-k^2$. In a similar fashion, we can use the boundary condition (\ref{bc4}) and the jump condition to obtain
\begin{eqnarray*}
\sqrt{n}\mathrm{D}^\top_{k\sqrt[4]{n}}[\varphi_w](x)+\frac{\sqrt{n}}{2}\varphi_w(x)-\sqrt{n}\mathrm{D}^\top_{\mathrm{i}k\sqrt[4]{n}}[\psi_w](x)-\frac{\sqrt{n}}{2}\psi_w(x)-\mathrm{D}^\top_{k}[\varphi_v](x)\notag\\
-\frac{1}{2}\varphi_v(x)+\mathrm{D}^\top_{\mathrm{i}k}[\psi_v](x)+\frac{1}{2}\psi_v(x)=0,\quad x\in \partial D.
\label{equation4}
\end{eqnarray*}
The four boundary integral equations (\ref{equation1}), (\ref{equation2}), (\ref{equation3}), and (\ref{equation4}) can be written as the system
\begin{eqnarray*}
\begin{pmatrix}
\mathrm{S}_{k\sqrt[4]{n}} & \mathrm{S}_{\mathrm{i}k\sqrt[4]{n}} & -\mathrm{S}_{k} & -\mathrm{S}_{\mathrm{i}k}\\
\mathrm{D}^\top_{k\sqrt[4]{n}}+\frac{1}{2}I & \mathrm{D}^\top_{\mathrm{i}k\sqrt[4]{n}}+\frac{1}{2}I & -\mathrm{D}^\top_{k}
-\frac{1}{2}I & -\mathrm{D}^\top_{\mathrm{i}k}-\frac{1}{2}I\\
\sqrt{n}\mathrm{S}_{k\sqrt[4]{n}} & -\sqrt{n}\mathrm{S}_{\mathrm{i}k\sqrt[4]{n}} & -\mathrm{S}_{k} & \mathrm{S}_{\mathrm{i}k}\\
\sqrt{n}\mathrm{D}^\top_{k\sqrt[4]{n}}+\frac{\sqrt{n}}{2}I & -\sqrt{n}\mathrm{D}^\top_{\mathrm{i}k\sqrt[4]{n}}-\frac{\sqrt{n}}{2}I & -\mathrm{D}^\top_{k}
-\frac{1}{2}I & \mathrm{D}^\top_{\mathrm{i}k}+\frac{1}{2}I
\end{pmatrix}
\begin{pmatrix}
\varphi_w\\
\psi_w\\
\varphi_v\\
\psi_v
\end{pmatrix}=
\begin{pmatrix}
0\\
0\\
0\\
0
\end{pmatrix},
\end{eqnarray*}
where $I$ denotes the identity operator. After discretization with the boundary element collocation method (see \cite[Section 4.3]{kleefeldhot} for details) and the use of the nonlinear eigenvalue solver by Beyn \cite{beyn2012integral}, we are able to determine the transmission eigenvalues $k$ for a given constant $n$. 

\subsection{Numerical results}
We first use the separation of variables approach to compute the clamped transmission eigenvalues for a unit disk using the index of refraction $n=10$ and $n=100$. The first ten real-valued eigenvalues not counting multiplicities are reported to high accuracy in Table \ref{unitdisk} using various $p\in \mathbb{N}_0$.

\begin{table}[!ht]
\centering
\caption{\label{unitdisk}First ten real-valued transmission eigenvalues (not counting multiplicities) using the index of refraction $n=10$ and $n=100$ for a unit disk. In parenthesis are listed the corresponding Bessel indices $p\in \mathbb{N}_0$.}
\begin{tabular}{c|c|c}
   $k$  & $n=10$  & $n=100$\\
   \hline
   1. & 4.429820190009274 (1) & 1.889337846549858 (0)\\
   2. & 4.907103277141769 (0) & 2.270739046479443 (1)\\
   3. & 4.915126540867455 (2) & 2.648103084653914 (2)\\
   4. & 5.523429130292435 (3) & 3.019471264958882 (3)\\
   5. & 6.149185323266162 (4) & 3.085608474133166 (0)\\
   6. & 6.777920075835243 (5) & 3.373887214452198 (1)\\
   7. & 7.405777394866369 (6) & 3.385829288328941 (4)\\
   8. & 8.031556535451719 (7) & 3.748114844619391 (5)\\
   9. & 8.629637722213953 (0) & 3.767023646982871 (2)\\
   10.& 8.654922094217133 (8) & 4.107049944231099 (6)
\end{tabular}
\end{table}
Note that all eigenvalues with index $p\neq 0$ have multiplicity two. Additionally, we see the monotonicity with respect to $n>1$ for the first transmission eigenvalue, as shown theoretically in Theorem \ref{monotrans}. Interestingly, we observe monotonicity for all real-valued transmission eigenvalues for the unit disk. 
Note that we are also able to find purely complex-valued interior transmission eigenvalues, although the existence has not been proven yet. For example, for $n=10$ we obtain the eigenvalues $3.566132162943008\pm0.684076560127309\mathrm{i}$ (0), $5.318791306918498\pm 0.499795992339707\mathrm{i}$ (1), and $6.497129223785265\pm 0.748001142808220\mathrm{i}$ (2) and many more, which are closely situated towards the real axis. 
Additionally, we show some results for the case $0< n< 1$. We pick some selected indices of refraction such as $n\in \{1/128,1/64,1/32,1/16,1/8\}$ and compute the first four real-valued transmission eigenvalues for the unit disk. 

\begin{table}[!ht]
\centering
\caption{\label{secondcase}First four real-valued transmission eigenvalues (not counting multiplicities) using the index of refraction $n\in \{1/128,1/64,1/32,1/16,1/8\}$ for a unit disk. In parenthesis are listed the corresponding Bessel indices $p\in \mathbb{N}_0$.}
\begin{tabular}{l|c|c|c|c}
$n$ & 1. TE & 2. TE & 3. TE & 4. TE\\
\hline
1/128 & 5.9585 (0) &7.1723 (1) &8.3679 (2) &9.5433 (3)\\
1/64  & 6.0190 (0) &7.2031 (1) &8.3902 (2) &9.5617 (3)\\
1/32  & 6.1756 (0) &7.2717 (1) &8.4383 (2) &9.6005 (3)\\
1/16  & 7.4522 (1) &8.1199 (0) &8.5523 (2) &9.6890 (3)\\
1/8   & 9.7142 (1) &8.9211 (0) &8.9299 (2) &9.9369 (3)\\
\end{tabular}
\end{table}

As we can see in Table \ref{secondcase}, we observe the monotonicity with respect to $n>1$ for the first real-valued transmission eigenvalue as proven in Theorem \ref{monotrans}. Interestingly, we observe the same monotonicity behavior for the other eigenvalues as well.

Next, we compute the real-valued transmission eigenvalues with boundary integral equations using the boundary element collocation method with $120$ collocation points and $N=96$ quadrature points for Beyn's method to numerically approximate the two contour integrals. In Table \ref{elli}, we list the first five real-valued transmission eigenvalues using the refraction index $n=10$.

\begin{table}[!ht]
\centering
\caption{\label{elli}First five real-valued transmission eigenvalues (TE) including multiplicities for a unit disk and various ellipses with half-axis $a=1$ and $b\in \{0.95, 0.9, 0.85, 0.8\}$ using the index of refraction $n=10$.}
\begin{tabular}{l|c|c|c|c|c}
     $(1,b)$ & 1. TE &  2. TE & 3. TE & 4. TE & 5. TE\\
     \hline
     (1,1)   &4.4298&4.4298&4.9071&4.9151&4.9151\\
     (1,0.95)&4.4948&4.6107&4.9516&5.0502&5.1347\\
     (1,0.9) &4.5837&4.8233&5.0045&5.2124&5.3881\\
     (1,0.85)&4.7051&5.0717&5.0741&5.4073&5.6757\\
     (1,0.8) &4.8701&5.1669&5.3603&5.6408&6.0025
\end{tabular}
\end{table}
As we observe in the first row of Table \ref{elli}, we obtain a five-digit accuracy for the unit disk compared to Table \ref{unitdisk}. In addition, we list the transmission eigenvalues for various ellipses with half axes $a=1$ and $b\in \{0.95, 0.9, 0.85, 0.8\}$. We observe a nice monotonicity behavior with respect to the parameter $b$ for all the listed TEs. The smaller $b$, the smaller the area of the ellipse and the larger the TEs.

Finally, we also list five real-valued transmission eigenvalues for the deformed ellipse given parametrically by 
$$\big(0.75\cdotp \cos(t)+\epsilon\cdotp \cos(2\cdotp t),\sin(t)\big)^\top \quad t\in [0,2\pi]$$ 
using $\epsilon \in \{0.1,0.2,0.3\}$. The results are reported in Table \ref{more}.
\begin{table}[!ht]
\centering
\caption{\label{more}First five real-valued transmission eigenvalues (TE) for the deformed ellipses with $\epsilon\in \{0.1, 0.2, 0.3\}$ using the index of refraction $n=10$.}
\begin{tabular}{l|c|c|c|c|c}
     $\epsilon$ & 1. TE &  2. TE & 3. TE & 4. TE & 5. TE\\
     \hline
     0.1   &5.1758&5.3428&5.7179&5.9858&6.3583\\
     0.2   &5.3462&5.5014&5.7695&6.1152&6.3883\\
     0.3   &5.5495&5.7166&5.8588&6.2672&6.4687
\end{tabular}
\end{table}
As we observe, we obtain a nice monotonicity with respect to the parameter $\epsilon$. Precisely, the larger $\epsilon$, the larger the TEs. However, the area of the deformed ellipse remains constant.

\section{Summary and outlook}\label{summary}
In this paper, we analyzed the transmission eigenvalue problem for biharmonic scattering with a penetrable obstacle. Starting from the direct biharmonic scattering formulation, we derived the corresponding biharmonic eigenvalue problem and established its fundamental spectral properties. In particular, we proved the discreteness of the spectrum, the existence of infinitely many real transmission eigenvalues, and the monotonic dependence of the first eigenvalue on the refractive index. These results extend the well-studied acoustic theory to the biharmonic setting, which naturally arises in Kirchhoff–Love plate models and other elastic structures. Numerical experiments, based on separation of variables and boundary integral equation methods, confirmed the theoretical results and illustrated the monotonicity for various geometries.

This work highlights the role of biharmonic transmission eigenvalues as a valuable spectral tool for inverse problems in elasticity and related applications. Now that a rigorous transmission eigenvalue problem for biharmonic scattering has been established, it would be natural to explore qualitative reconstruction methods such as the factorization method, the linear sampling method, or the direct sampling method in this setting. \\

\noindent{\bf Acknowledgments:} The research of I. Harris is partially supported by the NSF DMS Grant 2208256 and 2509722.\\

\noindent{\bf Disclosure:} No potential conflict of interest was reported by the author(s).


\end{document}